\documentclass[a4paper,12pt,reqno]{amsart}
\usepackage{amsmath}
\usepackage{amssymb,amsthm,amsfonts,bm}
\usepackage[colorlinks=true, urlcolor=blue, linkcolor=red, citecolor=green]{hyperref}
\usepackage[paperwidth=210mm,paperheight=297mm,inner=2.8cm,outer=2.8cm,top=2.6cm,bottom=2.6cm]{geometry}
\usepackage{indentfirst}

\def\CC{\mathbb{C}}

\def\NN{\mathbb{N}}

\def\PP{\mathbb{P}}
\def\QQ{\mathbb{Q}}
\def\RR{\mathbb{R}}

\def\mb{\mathbf}

\def\bt{\mathbf{t}}

\def\bt{\mathbf{t}}

\newtheorem{thm}{Theorem}[section]
\newtheorem{prop}[thm]{Proposition}
\newtheorem{lem}[thm]{Lemma}
\newtheorem{cor}[thm]{Corollary}

\newtheorem{rem}[thm]{Remark}
\let\oldrem\rem
\renewcommand{\rem}{\oldrem\normalfont}

\newtheorem{exa}[thm]{Example}
\let\oldexa\exa
\renewcommand{\exa}{\oldexa\normalfont}

\begin{document}

\title{Almost-toric Hypersurfaces}
\author{Bo Lin}
\address{Department of Mathematics, University of California, CA 94720, USA}
\email{linbo@math.berkeley.edu}
\
\thanks{The author thanks Bernd Sturmfels for his guidance, Ralph Morrison for his insightful suggestions and
Nathan Ilten for discussions. This work was supported by the Thematic Program of National Institute for
Mathematical Sciences, Daejeon, Korea, which hosted the author in the summer of 2014}

\begin{abstract}
  An almost-toric hypersurface is parameterized by monomials multiplied by polynomials in one extra
  variable.
  We determine the Newton polytope of such a hypersurface, and apply this to give an
  algorithm for computing the implicit equation.
\end{abstract}

\maketitle

\section{Introduction and Formula} \label{s:intro}

Toric varieties are parameterized by monomials. They
form an important and rich class of examples in algebraic geometry and often provide a testing ground for
theorems \cite{CLS11}. Here we work with toric varieties that need not be normal.
We fix an algebraically closed field $K$, and we set
$K^{*}=K \backslash \{0\}$. Our first ingredient is an arbitrary projective toric
variety of codimension $2$ in $\mathbb{P}^{n+1}$. This is defined as follows.

Fix an $n\times (n+2)$ integer matrix $A$ whose columns span the lattice $\mathbb{Z}^n$:
\[
A\,\,=\,\,\begin{bmatrix}
\mb{a}_0 & \mb{a}_1 & \ldots & \mb{a}_{n+1} \\
\end{bmatrix}.
\]
The column vectors $\mb{a}_i$ correspond to Laurent monomials
$\bt^{\mb{a}_i}=t_{1}^{a_{1,i}}t_{2}^{a_{2,i}}\ldots t_{n}^{a_{n,i}}$
 in the variables  $\bt=(t_1,t_2,\ldots,t_{n})$.
 These $n+2$ monomials specify a monomial map
\[
\Phi_{A}: (K^*)^n \rightarrow (K^*)^{n+2}, \,
 \bt\to (\bt^{\mb{a}_0},\bt^{\mb{a}_1},\ldots,\bt^{\mb{a}_{n+1}}).
\]
Throughout this paper we assume that all columns of $A$ sum to the
same positive integer $d$. Under this hypothesis, we obtain an induced  map
$\, \Phi_{A}: (K^*)^n \rightarrow  \PP^{n+1} $. The toric variety $X_A$ is the
closure in $\PP^{n+1}$ of the image of $\Phi_A$.
The degree of $X_A$ is the normalized volume of the polytope ${\rm conv}(A)$, and
the equations defining $X_A$ form the {\em toric ideal},
a well-studied object in combinatorial commutative algebra
and its applications \cite{Stu96}.

A natural extension of toric theory is the study of
{\em complexity one T-varieties} \cite{IS11}.
There are varieties with an action of a torus whose
general orbits have codimension one. They can be
viewed as a family of (possibly reducible) toric varieties over a curve. Our aim here
is to explore such T-varieties from the point of view
of symbolic computation. For simplicity we assume
that T-varieties are rational, i.e. the underlying curve is rational, and we focus on
projective hypersurfaces.

Our second ingredient is a vector of univariate polynomials in a new variable $x$:
\begin{equation}\label{def:Y}
{\bf f} \,= \,
\bigl(f_{0}(x),f_{1}(x),\ldots,f_{n+1}(x)\bigr) \,\, \in \,K[x]^{n+2}.
\end{equation}
This vector specifies a parametric curve $Y_{\mathbf{f}} \subset \PP^{n+1}$,
namely the closure of the set of points
$(f_{0}(x):f_{1}(x):\ldots:f_{n+1}(x))$.
Let $Z_{A,{\bf f}}$ denote the
 \emph{Hadamard product} in $\PP^{n+1}$ of the toric variety $X_A$ and the curve $Y_\mathbf{f}$.
 By definition, this is the Zariski closure of the~set
\[
\bigl\{(\bt^{\mb{a}_0}f_{0}(x):\bt^{\mb{a}_1}f_{1}(x):\ldots:\bt^{\mb{a}_{n+1}}f_{n+1}(x))\in \PP^{n+1}
\,|\,\,\bt\in (K^{*})^{n},x\in K \bigr\} \,\, \subset \,\, \PP^{n+1}.
\]
Under some mild hypotheses (see Theorem \ref{thm:main}(a)),
the variety  $Z_{A,{\bf f}}$ has codimension $1$,
and we call it the {\em almost-toric hypersurface}
associated with  $(A,{\bf f})$. We shall present a fast
method for implicitizing $Z_{A,{\bf f}}$. The output of our algorithm
is the irreducible polynomial
in $K[u_0,u_1,\ldots,u_{n+1}]$ that vanishes on $Z_{A,{\bf f}}$.
The torus action given by $A$ ensures
 that its Newton polytope ${\rm Newt}(Z_{A,{\bf f}}) $
lies in a plane in $ \RR^{n+2}$, so it is a polygon.
Our first main result is the following
combinatorial formula for this Newton polygon.

The ingredients in our formula are two matrices, which we now define.
The {\em Pl\"{u}cker matrix} associated with $A$ is the
$(n+2) \times (n+2)$-matrix $P_A = (p_{ij}) $ with entries
\[
p_{ij}=\begin{cases}
\frac{1}{\delta} (-1)^{i+j}\det(A_{[i,j]}), & i<j; \\
-p_{ji}, & i>j; \\
0, & i=j,
\end{cases}
\]
where $\delta$ is the greatest common divisor of all $\det(A_{[i,j]})$.
Here $A_{[i,j]}$ is the $n \times n$ submatrix of $A$
obtained by deleting the columns ${\bf a}_i$ and ${\bf a}_j$.
The Pl\"ucker matrix $P_A$ is skew-symmetric and has rank $2$,
and its row space is the kernel of $A$. The latter property implies that
all rows and all columns of $P_A$ sum to zero.

The {\em valuation matrix} associated with ${\bf f}$ is
an integer matrix $V_{\bf f}$ with $n+2$ rows that is defined as follows.
Since $K$ is algebraically closed, each of our polynomials
$f_i(x)$ factors into linear factors in $K[x]$.
Let $g_{1}(x),g_{2}(x),\ldots,g_{m}(x)$ be the list of all distinct linear factors of
$f_0(x) f_1(x) \ldots f_{n+1}(x)$.
We write ${\rm ord}_{g_j}f_i$ for the order of vanishing of
$f_i(x)$ at the unique root of $g_j(x)$.
We organize these numbers into the vectors
\[
\mb{u}_{j}=(\text{ord}_{g_{j}}f_{0},\text{ord}_{g_{j}}f_{1},\ldots,\text{ord}_{g_{j}}f_{n+1})\in \NN^{n+2}
\qquad \hbox{ for $1\le j\le m$.}
\]
We now aggregate these vectors according to the lines they span.
Let $S=\{\mb{u}_{1},\ldots,\mb{u}_{m}\}$. If two vectors in $S$ are linearly dependent, then we delete them
and add their sum to the set. We repeat this procedure. After finitely many steps, we end up with
a new set $\,S'=\{\mb{v}_{1},\mb{v}_{2},\ldots,\mb{v}_{\ell}\}\,$
whose vectors span distinct lines. The {\em valuation matrix} is
\[
V_{\bf f} \,\,= \,\,
\begin{bmatrix}\,
\mb{v}_{1}^{T} & \mb{v}_{2}^{T} & \ldots &\mb{v}_{l}^{T} & (-\sum_{j=1}^{l}{\mb{v}_{j}})^{T}
\end{bmatrix}.
\]
The last vector represents the valuation at $\infty$.
It ensures that the rows of $V_{\bf f}$ sum to zero.
The following theorem allows us to derive the Newton polygon
from $A$ and ${\bf f}$.

\begin{thm} \label{thm:main}
The edges of the Newton polygon of $Z_{A,{\bf f}}$ are the columns of the product of
the Pl\"{u}cker matrix $P_A$ and the valuation matrix $V_{\bf f}$.
More precisely, given $A$ and ${\bf f}$,
\begin{itemize}
\item[(a)] if $\,{\rm rank}(P_A \cdot V_{\bf f}) = 0\,$ then $Z_{A,{\bf f}}$ is not a hypersurface;
\item[(b)] if $\,{\rm rank}(P_A \cdot V_{\bf f}) = 1\,$ then $Z_{A,{\bf f}}$ is a toric hypersurface;
\item[(c)] if $\,{\rm rank}(P_A \cdot V_{\bf f}) = 2\,$ then $Z_{A,{\bf f}}$ is a hypersurface but not
    toric.
The directed edges of the Newton polygon of $Z_{A,{\bf f}}$ are the nonzero column vectors of
$P_{A} \cdot V_{\bf f}$.
\end{itemize}
\end{thm}

The rest of this paper is organized as follows.
In Section \ref{s:prf} we present the proof
of Theorem \ref{thm:main}, and we illustrate
this result with several small examples.
Earlier work of Philippon and Sombra \cite[Proposition 4.1]{PS08} yields
an expression for the degree of the almost-toric hypersurface
$Z_{A,{\bf f}}$ as a certain sum of integrals over ${\rm conv}(A)$.

Section \ref{s:alg} is concerned with computational issues.
Our primary aim is to give a fast algorithm for computing the
implicit equation of the almost-toric hypersurface $Z_{A,{\bf f}}$.
We develope such an algorithm and implement it in {\tt Maple 17}. A case
study of hard implicitization problems demonstrate
that our method performs very well.

The mathematics under the hood of Theorem \ref{thm:main} is
 {\em tropical algebraic geometry} \cite{MS14}.
 The proof relies on a technique known as
{\em tropical implicitization} \cite{STY07, ST08, SY08}.
Thus this article offers a concrete demonstration
that tropical implicitization can serve as an efficient
and easy-to-use tool  in computer algebra. It lays the
foundation for future work that will extend toric algebra \cite{Stu96}
and its numerous applications to the almost-toric setting
of complexity one T-varieties \cite{IS11}.

\vfill

\section{Proof,  Examples, and the Philippon-Sombra Formula }\label{s:prf}
In this section we prove Theorem \ref{thm:main}. In our proof we use the technique of
\emph{tropicalization}. We then present examples to illustrate our main theorem. We also consider an easier
task: finding the degree of the almost-toric hypersurface. We compare our result to existing results,
including the Philippon-Sombra Formula in \cite[Proposition 1.2]{PS08} and another formula in \cite{SY08}.

First we prove some of the claims made in Section \ref{s:intro}.

\begin{lem}\label{lem:gon}
Let $Z_{A,{\bf f}}$ be an almost-toric hypersurface. Then ${\rm Newt}(Z_{A,{\bf f}}) $ is at most
$2$-dimensional in $\RR^{n+2}$.
\end{lem}

\begin{proof}
Substituting variables $u_{0},\ldots,u_{n+1}$ by the parameterization $u_{i}=\bt^{\mb{a}_i}f_{i}(x)$ in the
implicit equation
$p(u_{0},\ldots,u_{n+1})$ of $Z_{A,{\bf f}}$, we get another polynomial in variables
$t_{1},t_{2},\ldots,t_{n},x$. The latter polynomial is the zero polynomial. After this substitution, each term
in $p$ becomes the product of a monomial in variables $t_{1},t_{2},\ldots,t_{n}$ and a polynomial in $x$.
Since $p$ is the generator of the principal ideal corresponding to the almost-toric hypersurface, all such
monomials in variables $t_{1},t_{2},\ldots,t_{n}$ are the same; otherwise the almost-toric hypersurface would
vanish on a polynomial that contains some terms of $p$, a contradiction. Suppose after substitution the
monomial is $ \prod_{i=1}^{n}{t_{i}^{\alpha_{i}}}$. If ${\bf v}=(v_{0},\ldots,v_{n+1})$ is a vertex of ${\rm
Newt}(Z_{A,{\bf f}}) $, then $p$ contains a term $\prod_{i=0}^{n+1}{u_{i}^{v_{i}}}$. Therefore $ {\bf v}\cdot
A^{T}=(\alpha_{1},\ldots,\alpha_{n})$. So vertices of ${\rm Newt}(Z_{A,{\bf f}}) $ satisfy $n$ independent
linear equations and we conclude that ${\rm Newt}(Z_{A,{\bf f}}) $ is at most $2$-dimensional.
\end{proof}

\begin{rem}
Even if $Z_{A,{\bf f}}$ is a hypersurface, its Newton polygon could be a degenerate
polygon that has only two vertices.
\end{rem}

\begin{lem}
Let $P_A$ be a Pl\"{u}cker matrix. Then $P_{A}$ is skew-symmetric. The rank of $P_{A}$ is $2$ and the entries
in each row and column of $P_{A}$ sum to $0$.
\end{lem}

\begin{proof}
The first claim follows directly from the definition of $P_{A}$. For the second claim, note that the inner
product of the $i$-th row of $P_{A}$ and the $j$-th row of $A$ is the determinant of the following $(n+1)\times
(n+1)$ matrix up to sign: append the $j$-th row of $A$ to matrix $A$ and then delete the $i$-th column of the new
matrix. Therefore ${\rm row}(P_{A})\subseteq {\rm row}(A)^{\perp}$, which means the rank of $P_{A}$ is at most
$2$. Since $A$ has full rank, by the definition of $P_{A}$ there is a nonzero non-diagonal entry $p_{ij}$ of
$P_{A}$. By the first claim, $p_{ji}$ is nonzero too. So we get a nonzero $2\times 2$ minor in $P_{A}$. For
the last claim, since $P_{A}$ is
skew-symmetric it is enough to prove the claim for rows. It turns out that up to sign, the sum of entries in the
$i$-th row is the determinant of the matrix formed by $A_{[i]}$ and the vector ${\bf 1}=(1,1,\ldots,1)\in
K^{n+1}$, where $A_{[i]}$ is the matrix obtained from $A$ by deleting the $i$-th column. Since each column of
$A$ has sum $d$, the vector ${\bf 1}$ lies in the row space of $A_{[i]}$, so the matrix is singular and its
determinant
is zero.
\end{proof}

Next, we explore $Z_{A,{\bf f}}$. We consider its \emph{tropicalization} (we follow the definition from
\cite[Definition 3.2.1]{MS14}). Let $X_{A},Y_{\bf f},Z_{A,{\bf f}}$ be defined as in Section \ref{s:intro}.
Note that $Z_{A,{\bf f}}$ is the Hadamard product $\overline{X_{A}*Y_{\bf f}}$ of the varieties $X_{A}$ and
$Y_{\bf f}$. We have the following result relating the tropicalizations of $X_{A},Y_{\bf
f},Z_{A,{\bf f}}$:
\begin{prop}\cite[Corollary 3.3.6]{Cue10}\label{prop:ha}
\[
{\rm trop}(Z_{A,{\bf f}})={\rm trop}(X_{A})+{\rm trop}(Y_{\bf f}),
\]
where the sum is Minkowski sum.
\end{prop}

So in order to find $\text{trop}(Z_{A,{\bf f}})$, it is enough to find
$\text{trop}(X_{A})$ and $\text{trop}(Y_{\bf f})$.
\begin{lem}\label{lem:X}
If ${\rm row}(A)$ is the row space of $A$ with real coefficients, then
\[
{\rm trop}(X_{A})={\rm row}(A).
\]
\end{lem}

\begin{proof}
By Proposition \ref{prop:ha}, it is enough to prove the case when $n=1$ and then use induction. For $n=1$ we
need to show that if $a_{0},a_{1},\ldots,a_{n+1}$ are integers and
\[
X_{A}=\text{cl}(\{(t^{a_{0}}:t^{a_{1}}:\ldots:t^{a_{n+1}})\in (\PP^{*})^{n+1}|t\in K^{*}\}),
\]
then
\[
\text{trop}(X_{A})=\{r\cdot{\bf a}|r\in \RR\},
\]
where ${\bf a}=(a_{0},a_{1},\ldots,a_{n+1})$. In this case the ideal $I(X_{A})$ is generated by binomials as
follows (cf. \cite[Corollary 4.3]{Stu96}):
\[
I(X_{A})=\langle {\bf x}^{\bf u}-{\bf x}^{\bf v}|{\bf a\cdot u=a\cdot v}\rangle.
\]
Then all points in $\text{trop}(X_{A})$ are scalar multiples of ${\bf a}$.
\end{proof}

\begin{lem}\label{lem:Y}
Let $Y_{\bf f}$ be defined as in (\ref{def:Y}), and $S=\{g_{1},\ldots,g_{m},\infty\}$. Then
\begin{equation}\label{eq:Y}
{\rm trop}(Y_{\bf f})=\bigcup_{z\in
S}{\{\lambda(\text{ord}_{z}{f_{0}},\text{ord}_{z}{f_{1}},\ldots,\text{ord}_{z}{f_{n+1}})\in \mathbb{R}^{n+2}
|\lambda \ge 0\}}.
\end{equation}
In addition, ${\rm trop}(Y_{\bf f})$ is an $1$-dimensional balanced polyhedral fan in $\RR^{n+2}$ and the rays
are spanned by the vectors ${\bf v}_{z}$, where
${\bf v}_{z}=(\text{ord}_{z}{f_{0}},\text{ord}_{z}{f_{1}},\ldots,\text{ord}_{z}{f_{n+1}})$.
\end{lem}

\begin{proof}
Let $K'=K\{\!\{t\}\!\}$ be the field of \emph{Puiseux series}\cite[Example 2.1.3]{MS14} in variable $t$ with coefficients in $K$. Then $K'$ has a nontrivial valuation and is algebraically closed. Let $Y'_{\bf f}$ be a variety parameterized as in (\ref{def:Y}) but $x$ varies in $K'$ instead. Note that $Y_{\bf f}$ and $Y'_{\bf f}$ have the same ideal, so $\text{trop}(Y_{\bf f})=\text{trop}(Y'_{\bf f})$.

By \cite[Theorem 3.2.3]{MS14}, $\text{trop}(Y'_{\bf f})=\overline{\text{val}(Y'_{\bf f})}$, where
\[
\text{val}(Y'_{\bf f})=\{(\text{val}(f_0(x)),\ldots,\text{val}(f_{n+1}(x)))|x\in K'\}.
\]
Since $\QQ$ is dense in $\RR$, the right hand side of (\ref{eq:Y}) is the closure of
\[
B=\bigcup_{z\in S}{\{\lambda(\text{ord}_{z}{f_{0}},\text{ord}_{z}{f_{1}},\ldots,\text{ord}_{z}{f_{n+1}})\in
\mathbb{R}^{n+2} |\lambda\in \mathbb{Q^{+}}\}}.
\]
It then suffices to show $B=\text{val}(Y'_{\bf f})$. We first show that $B\subseteq \text{val}(Y'_{\bf f})$. Fixing $z\in S$ and $\lambda>0$, we get a vector
\[{\bf u}=\lambda(\text{ord}_{z}{f_{0}},\text{ord}_{z}{f_{1}},\ldots,\text{ord}_{z}{f_{n+1}})\in B.\]
\begin{itemize}
  \item If $z\ne \infty$, then $z=g_{j}$ for some $1\le j\le m$. Since each $g_{j}$ is linear, we may assume $g_{j}(x)=x-r_j$ where $r_j\in K$. For each $i$, we have $f_{i}(x)=(x-r_{j})^{\text{ord}_{g_{j}}{f_{i}}}h_{i}(x)$, where $h_{i}(r_j)\ne 0$. Then $\text{val}(h_{i}(r_j))=0$. Now if we take $x=r_{j}+t^{\lambda}\in K'$, then $f_{i}(x)=t^{\lambda\text{ord}_{g_{j}}{f_{i}}}h_{i}(r_{j}+t^{\lambda})$. Notice that $h_{i}(x)$ is a
  polynomial in $K[x]$. Then $t^{\lambda}$ divides
  $h_{i}(r_{j}+t^{\lambda})-h_{i}(r_j)$ in $K'$, so $\text{val}(h_{i}(r_{j}+t^{\lambda})-h_{i}(r_j))>0$. Then
  $\text{val}(h_{i}(r_{j}+t^{\lambda}))=0$. Thus $\text{val}(f_{i}(x))=\lambda\text{ord}_{g_{j}}{f_{i}}$, which means
  ${\bf u}\in \text{val}(Y'_{\bf f})$.

  \item If $z=\infty$, then $\text{ord}_{z}{f_{i}}=-\deg(f_{i})$. We take $x=t^{-\lambda}\in K'$. Then among all
      terms in $f_{i}(x)$, the term with smallest valuations is the leading term, because $\lambda>0$. Then
  \[
  \text{val}(f_{i}(x))=(-\lambda)\deg(f_{i})=\lambda\text{ord}_{\infty}{f_{i}}.
  \]
  So ${\bf u}\in \text{val}(Y'_{\bf f})$, too.
\end{itemize}

We next show that $\text{val}(Y'_{\bf f}) \subseteq B$. Suppose ${\bf u}\in \text{val}(Y'_{\bf f})$. Then there exists
$x_0\in K'$ such that ${\bf u}=(\text{val}(f_{0}(x_0)),\ldots,\text{val}(f_{n+1}(x_0)))$. We may assume that ${\bf u}\ne 0$. We must deal with two cases.
\begin{itemize}
  \item All terms of $x_0$ are nonnegative powers of $t$. Since
      $f_{i}(x)=\prod_{j=1}^{m}{g_{j}(x)^{\text{ord}_{g_{j}}{f_{i}}}}$, we have
  \[
  \text{val}(f_{i}(x_0))=\sum_{j=1}^{m}{\text{ord}_{g_{j}}{f_{i}}\cdot \text{val}(g_{j}(x_0))}.
  \]
      Note that all $g_j$ are linear functions, then for $1\le j<j'\le m$ we have $\text{val}(g_{j}(x_0)-g_{j'}(x_0))=0$. Hence for $1\le j\le m$ there is at most one nonzero $\text{val}(g_{j}(x_0))$, while they are not all zero because ${\bf u}\ne 0$. Suppose $\text{val}(g_{j}(x_0))=\lambda>0$. Then ${\bf u}=\lambda(\text{ord}_{z}{f_{0}},\ldots,\text{ord}_{z}{f_{n+1}})$ where $z=g_j\in S$ and ${\bf u}\in B$.
  \item At least one term of $x_0$ is a negative power of $t$. Suppose in $x_0$ the term with least degree is
  $ct^{-\frac{p}{q}}$, where $\frac{p}{q}\in \mathbb{Q}_{\ge 0}$. Then ${\bf u}=\frac{p}{q}(\text{ord}_{\infty}{f_{0}},\text{ord}_{\infty}{f_{1}},\ldots,\text{ord}_{\infty}{f_{n+1}})$
  is in $B$, too.
\end{itemize}

Finally, by \cite[Proposition 3.4.13]{MS14}, $\text{trop}(Y'_{\bf f})$ is $1$-dimensional and balanced.
\end{proof}

By the definition of $V_{\bf f}$, the tropicalization $\text{trop}(Y_{\bf f})$ is exactly the union of rays
generated by column vectors in $V_{\bf f}$. This implies the following corollary:
\begin{cor}\label{cor:trop} The tropicalization of $Z_{A,{\bf f}}$ is
\[
\{\mb{u}+\lambda\cdot \mb{v}^{T}|\mb{u}\in {\rm row}(A),\mb{v}\text{ is a column vector
of }V_{\bf f},\lambda\ge 0\}.
\]
\end{cor}

We now find ${\rm Newt}(Z_{A,{\bf f}}) $.
\begin{prop}\label{prop:direc}
The edges of ${\rm Newt}(Z_{A,{\bf f}}) $ are parallel to the nonzero column vectors in $P_{A}\cdot V_{\bf
f}$.
\end{prop}
\begin{proof}
By \cite[Proposition 3.1.10]{MS14}, $\text{trop}(Z_{A,{\bf f}})$ is the support of an $(n+1)$-dimensional
polyhedral fan, which is the \emph{$(n+1)$-skeleton} of the normal fan of ${\rm Newt}(Z_{A,{\bf f}}) $. Since
${\rm Newt}(Z_{A,{\bf f}})$ is a polygon, every cone in the $(n+1)$-skeleton of its normal fan is a cone
spanned by ${\rm row}(A)$ (which is exactly the orthogonal complement of the plane that contains the polygon)
and a ray inside this plane that is orthogonal to the corresponding edge of the polygon. By Corollary
\ref{cor:trop}, every one of these directed edges belongs to $\ker(A)$ and is orthogonal to the corresponding
column vectors of $V_{\bf f}$. Let ${\bf v}$ be a column vector of $V_{\bf f}$. Since $P_{A}$ is
skew-symmetric, we have
\[
{\bf v}^T \cdot P_{A} \cdot {\bf v}=0.
\]
Hence $P_{A}\cdot {\bf v}$ is orthogonal to ${\bf v}^T$, which means that there exists a scalar $c$ such that
$c(P_{A}\cdot {\bf v})$ represents an edge of ${\rm Newt}(Z_{A,{\bf f}})$.
\end{proof}

It remains to show that the length of column vectors in $P_A \cdot V_{\bf f}$ coincide with the edges of
${\rm Newt}(Z_{A,{\bf f}}) $. To analyze these lengths we recall the notion of \emph{multiplicity} of a
polyhedron which is maximal in a weighted polyhedral complex. We adopt the definition in \cite[Definition
3.4.3]{MS14}.
We have the following result:

\begin{lem}\cite[Lemma 3.4.6]{MS14}\label{lem:length}
The lattice length of any edge (defined as the number of lattice points on the edge minus $1$) of
${\rm Newt}(Z_{A,{\bf f}}) $ is the multiplicity of the corresponding
$(n+1)$-dimensional polyhedron in the normal fan of ${\rm Newt}(Z_{A,{\bf f}}) $.
\end{lem}

Note that if an edge of ${\rm Newt}(Z_{A,{\bf f}}) $ is expressed by a vector, then its lattice length
is the \emph{content} of that vector. Therefore it is enough to find out $\text{mult}(\sigma)$ for each
$\sigma$ in the $(n+1)$-skeleton of the normal fan of ${\rm Newt}(Z_{A,{\bf f}}) $. We cannot achieve this
directly from the definition of multiplicity, because it involves the implicit polynomial of $Z_{A,{\bf f}}$,
which is unknown to us. However, we can find out those multiplicities with the help of another result in
\cite{MS14}. The following proposition will lead to our main theorem.

\begin{prop}\label{prop:mult}
For every maximal cell $\sigma$ in the $(n+1)$-skeleton of the normal fan of ${\rm Newt}(Z_{A,{\bf f}}) $,
$\text{mult}(\sigma)$ is the content of the corresponding column vector of $P_A \cdot V_{\bf f}$.
\end{prop}

\begin{proof}
Let $\Sigma'$ be $\text{trop}(Z_{A,{\bf f}})$, which is a pure weighted polyhedral complex in $\RR^{n+2}$.
Suppose there are $m$ roots $z_{1},\ldots,z_{m}\in K$ of $\prod_{i=0}^{n+1}{f_{i}(z)}$. We define another pure
weighted polyhedral complex $\Sigma \subseteq \RR^{n+m+2}$ as follows:
\[
\Sigma=\{({\bf u,v})\in\RR^{n+m+2}|{\bf u}\in \text{row}(A),{\bf v}=\lambda\cdot e_{i},
1\le i\le m+1,\lambda\ge 0\},
\]
where $e_{i}$ is the vector with $i$-th component $1$ and others $0$ for $1\le i\le m$, and $e_{m+1}=-{\bf
1}$. Note that
\[
\Sigma=\text{trop}(Z),
\]
where $Z$ is the variety parameterized by
\[
\{((\bt^{\mb{a}_0}:\bt^{\mb{a}_1}:\ldots:\bt^{\mb{a}_{n+1}}),(x-z_{1},x-z_{2},\ldots,x-z_{m}))\in
\PP^{n+1}\times \CC^{m}|\,\,\bt\in (K^{*})^{n},x\in K\}.
\]

Note that we can get $\Sigma'$ from $\Sigma$ via a projection $\phi$, where $\phi$ keeps $u\in \text{row}(A)$
fixed and sends each $e_{i}$ to the transpose of the $i$-th column vector ${\bf v}_{i}$ in $V_{\bf f}$. So
$\phi$
maps a maximal cell $\sigma\in\Sigma $ to a maximal cell $\sigma'\in \Sigma'$. This correspondence of maximal
cells is a bijection. Suppose $\sigma'$ in the $(n+1)$-skeleton of the normal fan of ${\rm Newt}(Z_{A,{\bf
f}})$ corresponds to the vector ${\bf v}_{i}$. Then by \cite[(3.6.2)]{MS14}
\[
\text{mult}(\sigma')=\text{mult}(\sigma)\cdot[N_{\sigma'}:\phi(N_{\sigma})],
\]
where $N_{\sigma'}\subseteq \RR^{n+2}$ is the lattice generated by all integer points in the span of row
vectors of
$A$ and ${\bf v}_{i}^{T}$, and $N_{\sigma}\subseteq \RR^{n+m+2}$ is the lattice generated by all integer
points in the space ${\rm row}(A)\oplus \RR{\bf e}_{i}$. Let $N\subseteq \RR^{n+2}$ be the lattice generated
by row vectors of $A$ and ${\bf v}_{i}^{T}$. Then $N\subseteq N_{\sigma'},N_{\sigma}$ and
$[N_{\sigma'}:\phi(N_{\sigma})]=\frac{[N_{\sigma'}:N]}{[N_{\sigma}:N]}$. The lattice index $[N_{\sigma'}:N]$
is the greatest common divisor of all maximal minors of the matrix formed by all generating vectors of
$N_{\sigma'}$. Here this matrix is obtained from $A$ by adding one row vector ${\bf v}_{i}^{T}$ and its
maximal minors are the product of $\delta$ and one entry in the $i$-th column of $P_A \cdot V_{\bf f}$, so
$[N_{\sigma'}:N]$ is the product of $\delta$ and the content of the the $i$-th column vector of $P_A \cdot
V_{\bf f}$. Similarly, $[N_{\sigma}:N]$ is the greatest common divisor of all maximal minors of $A$, which is
$\delta$.

Finally the initial ideal $in_{\sigma}(I(Z))$ with respect to $\sigma$ of the ideal of $Z$ is the direct sum
of two prime ideals: the first one is toric, the second one is generated by linear polynomials. Then the quotient ring of $K[u_0,\ldots,u_{n+1}]$ modulo the toric ideal is a domain. Note that the quotient ring of $K[u_0,\ldots,u_{n+1},v_1,\ldots,v_m]$ modulo $in_{\sigma}(I(Z))$ is the quotient ring of $K[u_0,\ldots,u_{n+1}]$ modulo the toric ideal adjoining some variable(s) in $\{v_1,\ldots,v_m\}$, which is also a domain. Hence $in_{\sigma}(I(Z))$ is also prime. Then by the definition of multiplicity, $\text{mult}(\sigma)=1$. This finishes the proof.
\end{proof}

\begin{proof}[Proof of Theorem (\ref{thm:main})]
If $\,{\rm rank}(P_A \cdot V_{\bf f}) = 0\,$ then all column vectors of $V_{\bf f}$ belong to ${\rm row}(A)$,
which means $Z_{A,{\bf f}}$ contains the toric variety $X_{A}$ and has codimension $2$, so it is not a
hypersurface. If $\,{\rm rank}(P_A \cdot V_{\bf f}) = 1\,$, then essentially $\bf f$ provides one parameter
not appearing in $X_{A}$, which means that $Z_{A,{\bf f}}$ has codimension $1$ and is a toric hypersurface. If
$\,{\rm rank}(P_A \cdot V_{\bf f}) = 2\,$, then ${\rm Newt}(Z_{A,{\bf f}}) $ is a nondegenerate polygon and
with Proposition \ref{prop:direc}, Lemma \ref{lem:length} and Proposition \ref{prop:mult} we have proved
Theorem \ref{thm:main}.
\end{proof}

We illustrate Theorem \ref{thm:main} with the following example.
\begin{exa}\label{exa:H}
Let $Z_{A,{\bf f}}$ admit the following parameterization over $\CC$:
\[
(t_{1}^{2}(x^{2}+1):t_{1}t_{2}x^{3}(x-1):t_{1}t_{3}x(x+1):t_{2}^2(x-2)(x^2+1):t_{3}^{2}(x-1)^{2}(x+1)).
\]
In this example $A=\begin{bmatrix}
2&1&1&0&0 \\
0&1&0&2&0 \\
0&0&1&0&2
\end{bmatrix},d=2$. Then $P_{A}=\begin{bmatrix}
0&-2&2&1&-1 \\
2&0&-4&0&2 \\
-2&4&0&-2&0 \\
-1&0&2&0&-1 \\
1&-2&0&1&0
\end{bmatrix}$ and $\delta=2$. The linear factors of the univariate polynomials are $x,x-1,x+i,x-i,x+1,x-2$.
But we can combine $x\pm i$ into $x^2+1$. So the vectors are
\[
(0,3,1,0,0),(0,1,0,0,2),(2,0,0,2,0),(0,0,1,0,1),(0,0,0,1,0),(-2,-4,-2,-3,-3).
\]
Then the valuation matrix of ${\rm Newt}(Z_{A,{\bf f}})$ is
\[
V_{\bf f}=\begin{bmatrix}
0&0&2&0&0&-2 \\
3&1&0&0&0&-4 \\
1&0&0&1&0&-2 \\
0&0&2&0&1&-3 \\
0&2&0&1&0&-3 \\
\end{bmatrix}.
\]
Using ideal elimination in {\tt Macaulay 2} we can compute the implicit polynomial of ${\rm Newt}(Z_{A,{\bf
f}})$ in variables
$u_{0},u_{1},u_{2},u_{3},u_{4}$:
\begin{equation*}
\begin{split}
&16u_{1}^{4}u_{2}^{16}u_{3}^{2}-40u_{0}u_{1}^{4}u_{2}^{14}u_{3}^{2}u_{4}+8u_{0}^{2}u_{1}^{2}u_{2}^{14}u_{3}^{3}u_{4}-16u_{0}u_{1}^{6}u_{2}^{12}u_{3}u_{4}^{2}+20u_{0}^{2}u_{1}^{4}u_{2}^{12}u_{3}^{2}u_{4}^{2}\\
+&159u_{0}^{3}u_{1}^{2}u_{2}^{12}u_{3}^{3}u_{4}^{2}+u_{0}^{4}u_{2}^{12}u_{3}^{4}u_{4}^{2}+54u_{0}^{2}u_{1}^{6}u_{2}^{10}u_{3}u_{4}^{3}-77u_{0}^{3}u_{1}^{4}u_{2}^{10}u_{3}^{2}u_{4}^{3}+379u_{0}^{4}u_{1}^{2}u_{2}^{10}u_{3}^{3}u_{4}^{3}\\
+&5u_{0}^{2}u_{1}^{8}u_{2}^{8}u_{4}^{4}-27u_{0}^{3}u_{1}^{6}u_{2}^{8}u_{3}u_{4}^{4}-29u_{0}^{4}u_{1}^{4}u_{2}^{8}u_{3}^{2}u_{4}^{4}+163u_{0}^{5}u_{1}^{2}u_{2}^{8}u_{3}^{3}u_{4}^{4}-12u_{0}^{3}u_{1}^{8}u_{2}^{6}u_{4}^{5}\\
-&35u_{0}^{4}u_{1}^{6}u_{2}^{6}u_{3}u_{4}^{5}-425u_{0}^{5}u_{1}^{4}u_{2}^{6}u_{3}^{2}u_{4}^{5}+4u_{0}^{6}u_{1}^{2}u_{2}^{6}u_{3}^{3}u_{4}^{5}+87u_{0}^{5}u_{1}^{6}u_{2}^{4}u_{3}u_{4}^{6}+717u_{0}^{6}u_{1}^{4}u_{2}^{4}u_{3}^{2}u_{4}^{6}\\
+&103u_{0}^{6}u_{1}^{6}u_{2}^{2}u_{3}u_{4}^{7}-115u_{0}^{7}u_{1}^{4}u_{2}^{2}u_{3}^{2}u_{4}^{7}+12u_{0}^{7}u_{1}^{6}u_{3}u_{4}^{8}+4u_{0}^{8}u_{1}^{4}u_{3}^{2}u_{4}^{8}.\\
\end{split}
\end{equation*}
The vertices of Newton polygon of this implicit polynomial are
\[
(0,4,16,2,0),(2,8,8,0,4),(3,8,0,6,5),(7,6,0,1,8),(8,4,0,2,8),(4,0,12,4,2).
\]
The directed edges are
\[
\begin{split}
&(2,4,-8,-2,4),(-4,4,4,-2,-2),(-4,-4,12,2,-6) \\
&(1,-2,0,1,0),(4,-2,-6,1,3),(1,0,-2,0,1),
\end{split}
\]
and the product $P_{A}\cdot V_{\bf f}$ is
\[
\begin{bmatrix}
-4&-4&2&1&1&4\\
-4&4&4&-2&0&-2\\
12&4&-8&0&-2&-6\\
2&-2&-2&1&0&1\\
-6&-2&4&0&1&3\\
\end{bmatrix}.
\]
The column vectors of this matrix are the directed edges of ${\rm Newt}(Z_{A,{\bf f}})$.
\end{exa}

Next we compare our result with existing work. The Philippon-Sombra formula \cite[Proposition 4.1]{PS08}
computes the degree of $Z_{A,{\bf f}}$ from $A$ and $\bf f$. Let $B$ be a $2\times (n+2)$ matrix such that its
entries are integers and its row vectors span the kernel of $A$, and ${\bf b}_{0},\ldots,{\bf b}_{n+1}$ be the
column vectors of $B$. Given any column vector ${\bf v}=(v_{0},\ldots,v_{n+1})^{T}\in \RR^{n+2}$, we define a
polytope as the convex hull of the following $n+2$ vertices: ${\bf c}_{i}$ is the vector formed by the $i$-th
column vector of $A$ and $v_{i-1}$ for $1\le i\le n+2$. Then the vector $B\cdot {\bf v}$ admits a
triangulation $T$ of this polytope: for all $0\le i<j\le n+1$, the $n$-dimensional simplex formed by vertices
excluding ${\bf c}_{i},{\bf c}_{j}$ belongs to this triangulation if and only if the vector $B\cdot {\bf v}$
belongs to the nonnegative span of ${\bf b}_{i}$ and ${\bf b}_{j}$. Now we define a sum
\[
\partial_{\bf v}(A)=\frac{1}{\delta}\mathop{\sum}_{\sigma \in T}{|A_{\sigma}|\sum_{i\in \sigma}{v_{i}}}.
\]
Here $A_{\sigma}$ is the $n\times n$ minor of $A$ that contains all columns corresponding to $\sigma$.

\begin{prop}\label{prop:ps}
The degree of an almost-toric hypersurface $Z_{A,{\bf f}}$ is
\[
\mathop{\sum}_{{\bf v} \text{ is a column vector of } V_{f}}{\partial_{\bf v}(A)}.
\]
\end{prop}

This proposition is a direct corollary of \cite[Proposition 4.1]{PS08}.

\begin{rem}
For some ${\bf v}$, $\partial_{\bf v}(A)$ is negative, so the computation of degree using the Philippon-Sombra
formula is not very simple. The author attempted to obtain an alternative interpretation of the formula such
that all
summands are positive, but failed. However in Section \ref{s:alg} we present an algorithm to compute the
implicit polynomial of ${\rm Newt}(Z_{A,{\bf f}}) $, and in step $(3)$ we can compute the degree of this
polynomial efficiently.
\end{rem}

\cite[Theorem 5.2]{STY07} also provides an alternative way to compute the degree of $Z_{A,{\bf f}}$, using
tropical geometry. Applying this theorem to our almost-toric hypersurface $Z_{A,{\bf f}}$ we have the
following corollary.

\begin{cor}\label{cor:deg}
For a generic column vector ${\bf w}\in \RR^{n+2}$, the $i$-th coordinate of the vertex $face_{\bf w}({\rm
Newt}(Z_{A,{\bf f}}))$ is the number of intersection points, each counted with its \emph{intersection
multiplicity}, of the tropical hypersurface $\text{trop}(Z_{A,{\bf f}})$ with the half line ${\bf w}+\RR_{\ge
0}{\bf e}_{i}$.
\end{cor}

Here intersection multiplicity is defined in the remark after \cite[Theorem 5.2]{STY07} and ${\bf e}_{i}\in
\RR^{n+2}$ is the column vector with $i$-th component $1$ and others $0$ for $1\le i\le n+2$. From Corollary
\ref{cor:deg} we get the following proposition.

\begin{prop}\label{prop:deg}
Let $p(u_0,\ldots,u_{n+1})$ be the implicit polynomial of $Z_{A,{\bf f}}$. Then for a generic vector ${\bf
w}\in \RR^{n+2}$, the initial monomial $in_{\bf w}p$ is
\[
\mathop{\prod}_{(i,j)\in S}{u_{i}^{|{\bf e}_{i}P_{A}{\bf v}_{j}|}},
\]
where
\[
\begin{split}
S=&\{(i,j)|1\le i\le n+2,{\bf v}_{j} \text{ is a column vector of }V_{\bf f},\\
&{\bf e}_{i}^{T}P_{A}{\bf w},{\bf e}_{i}^{T}P_{A}{\bf v}_{j},{\bf v}_{j}^{T}P_{A}{\bf w}\text{ have the same sign.}\}.
\end{split}
\]

Hence the degree of $Z_{A,{\bf f}}$ is
\[
\mathop{\sum}_{(i,j)\in S}{|{\bf e}_{i}P_{A}{\bf v}_{j}|}.
\]
\end{prop}

\begin{proof}
If $\bf w$ is generic, then $\text{trop}(Z_{A,{\bf f}})$ and ${\bf w}+\RR_{\ge 0}{\bf e_i}$ have at most one
intersection point. Let ${\bf v}_{j}$ be the $j$-th column vector of matrix $V_{\bf f}$ and ${\rm
col}(A^{T})=\{{\bf u}^{T}|{\bf u}\in {\rm row}(A)\}$. Suppose there is an intersection point of the maximal
cone spanned by ${\rm row}(A)$ and the half line $\RR_{\ge 0}{{\bf v}_{j}}$ and ${\bf w}+\RR_{\ge 0}{\bf
e}_i$. Since $\bf w$ is generic, the $2$-dimensional subspace spanned by ${\bf w},{\bf e}_i$ has a unique
common point with a translation of a codimension $2$ subspace: ${\bf w}-{\rm col}(A^{T})$. Then there exists a
unique pair of nonzero (because $\bf w$ is generic) real numbers $\lambda_1,\lambda_2$ and ${\bf u}\in {\rm
row}(A)$ such that ${\bf u}^{T}+\lambda_{1}{\bf v}_{j}={\bf w}+\lambda_{2}{\bf e}_{i}$. Then the intersection
point exists if and only if $\lambda_1,\lambda_2>0$. Note that $\lambda_{1}>0$ if and
only if ${\bf v}_{j}$ and ${\bf w}$ are on the same side of the hyperplane spanned by ${\rm col}(A^{T})=\{{\bf
u}^{T}|{\bf u}\in {\rm row}(A)\}$ and ${\bf e}_{i}$ (if they don't span a hyperplane, then ${\bf e}_{i}\in
{\rm col}(A^{T})$ and no intersection point exists since $\bf w$ is generic). So $\det(\begin{bmatrix}{\bf
v}_{j} & {\bf e}_{i} & A^{T} \end{bmatrix})$ and $\det(\begin{bmatrix}{\bf w} & {\bf e}_{i} & A^{T}
\end{bmatrix})$ have the same sign.

Note that for any column vectors ${\bf a,b}\in \RR^{n+2}$, $\det(\begin{bmatrix}{\bf b}^{T} \\ {\bf a}^{T} \\
A \end{bmatrix})=\det(\begin{bmatrix}{\bf b} & {\bf a} & A^{T} \end{bmatrix})=\delta{\bf a}^{T}P_{A}{\bf b}$.
Hence $\lambda_{1}>0$ if and only if ${\bf e}_{i}^{T}P_{A}{\bf w}$ and ${\bf e}_{i}^{T}P_{A}{\bf v}_{j}$
have the same sign. Similarly, $\lambda_{2}>0$ if and only if ${\bf w}$ and ${\bf e}_{i}$ are on the
opposite sides of the hyperplane spanned by ${\rm col}(A^{T})$ and ${\bf v}_{j}$, which is equivalent to ${\bf
e}_{i}^{T}P_{A}{\bf v}_{j}$ and ${\bf v}_{j}^{T}P_{A}{\bf w}$ having the same sign. So $(i,j)\in S$ if and only
if there is an intersection point of the maximal cone spanned by ${\rm row}(A)$ and the half line $\RR_{\ge
0}{{\bf v}_{j}}$ with the half line ${\bf w}+\RR_{\ge 0}{\bf e_i}$. Next it suffices to show that the
intersection multiplicity of this point is $|{\bf e}_{i}P_{A}{\bf v}_{j}|$. By the definition of intersection
multiplicity, for this point it is the lattice index of the lattice spanned by ${\bf e}_{i},{\bf v}_{j}$ and
the transpose of row vectors of $A$, so the intersection multiplicity is $|{\bf e}_{i}P_{A}{\bf v}_{j}|$.
\end{proof}

\begin{rem}
Proposition \ref{prop:deg} enables us to compute the degree of $Z_{A,{\bf f}}$ without knowing ${\rm
Newt}(Z_{A,{\bf f}})$.
\end{rem}

\section{Algorithm, Implementation and Case Study}
\label{s:alg}

\subsection*{Algorithm to compute the implicit polynomial}
For an almost-toric hypersurface, we would like to compute its implicit polynomial in $n+2$ variables
$u_{0},u_{1},\ldots,u_{n+1}$ from $A$ and ${\bf f}$. An existing approach uses ideal elimination with
Gr\"{o}bner
bases, which is inefficient when $n$ is large. Based on Theorem \ref{thm:main} we have the following
alternative
approach:
\begin{enumerate}
  \item Compute $P_{A}$ from $A$, factorize $f_{0},f_{1},\ldots,f_{n+1}$ over $K$ into irreducible factors
      to get $V_{\bf f}$.
  \item Compute $P_{A}\cdot V_{\bf f}$ and verify it has rank $2$.
  \item Find ${\rm Newt}(Z_{A,{\bf f}}) $ using Theorem \ref{thm:main}.
  \item Determine all possible monomials in variables $u_{0},u_{1},\ldots,u_{n+1}$ that could appear in the
      implicit  polynomial.
  \item Use linear algebra to compute the coefficients of these monomials.
\end{enumerate}

We now explain our implementation of this method using the software {\tt Maple 17}. Among the five steps, the
first and second are trivial to implement ({\tt Maple 17} has the command {\it factor} which factors a
polynomial into irreducible factors over a given field).
\subsubsection*{Step (3)}
Theorem \ref{thm:main} tells us that the set of directed edges are the column vectors of $P_{A}\cdot V_{\bf
f}$. Then we need to arrange them in the correct order. We could project these vectors to a $2$-dimensional
space, by choosing two of the coordinates $1\le c_{1}<c_{2}\le n+2$. There is still the problem of
orientation: these directed edges admit two different arrangements. The correct orientation is determined by
the sign of the $(P_{A})_{c_{1},c_{2}}$.

Now suppose all directed edges are arranged in correct order and are the column vectors of a
matrix
\[
\begin{bmatrix}
c_{i,j}
\end{bmatrix}_{1\le i\le n+2,1\le j\le m}.
\]
If the vertex that corresponds to the first and $m$-th edges has coordinates $r_{1},\ldots,r_{n+2}$, then the
other
vertices have coordinates
\[
(r_{1}+\mathop{\sum}_{j=1}^{k}{c_{1,j}},r_{2}+\mathop{\sum}_{j=1}^{k}{c_{2,j}},\ldots,r_{n+2}+\mathop{\sum}_{j=1}^{k}{c_{n+2,j}}),k=1,\ldots
,m-1.
\]
We notice that for $1\le i\le n+2$, the $i$-th coordinate of the vector of vertices corresponds to the
exponents of $t_{i}$. Since the implicit polynomial is irreducible, the minimum of these exponents must be
$0$:
\[
\min_{1\le k\le m}\{r_{i}+\mathop{\sum}_{j=1}^{k}{c_{i,j}}\}=0, 1\le i\le n+2.
\]
Hence
\[
r_{i}=-\min_{1\le k\le m}\{\mathop{\sum}_{j=1}^{k}{c_{i,j}}\}, 1\le i\le n+2.
\]
So ${\rm Newt}(Z_{A,{\bf f}}) $ is uniquely determined by $P_{A}\cdot V_{\bf f}$ and we can compute its
vertices using the formula above.
\subsubsection*{Step (4)}
Given the vertices of ${\rm Newt}(Z_{A,{\bf f}}) $, we need to find all lattice points of this polygon.
Since all of them lie in the translation of a $2$-dimensional subspace, projection onto $2$ coordinates would
work. We find all lattice points within a convex polygon (which may be degenerate), then recover the
corresponding lattice points in ${\rm Newt}(Z_{A,{\bf f}}) $. These lattice points correspond to all possible
monomials in the implicit polynomial of $Z$: components of each vector are the exponents of variables
$u_{0},u_{1},\ldots,u_{n+1}$.\\

\subsubsection*{Step (5)}
Given all monomials of the implicit polynomial $p(u_0,\ldots,u_{n+1})$, it is enough to find the coefficients
of them. Consider the coefficients as undetermined unknowns. After the substitution
$u_{i}=\bt^{\mb{a}_i}f_{i}(x)$ we get another polynomial $q(t_1,\ldots,t_n,x)$. Each term in $q$ is the
product of a coefficient, a monomial in the variables $t_1,\ldots,t_n$ and a polynomial in the variable $x$.
We claim that all these monomials in variables $t_1,\ldots,t_n$ are the same. Assume the opposite situation.
Then we can pick all terms in $q$ with a particular monomial in variables $t_1,\ldots,t_n$ and get the
corresponding monomials in $p$, which form another polynomial $p'(u_0,\ldots,u_{n+1})$ with less terms than
$p$. Since $p$ is the implicit polynomial, polynomial $q$ must be identically zero, then $p'$ vanishes
everywhere too, so $p'$ belongs to the principal ideal of our hypersurface, a contradiction! So after the
substitution we can cancel the unique monomial in variables $t_1,\ldots,t_n$ and get a univariate polynomial
in $x$ with undetermined coefficients. Since this polynomial is identically zero, the undetermined
coefficients satisfy a system of homogeneous linear equations.\\

Next we use interpolation. Suppose there are $k$ possible monomials in the implicit polynomial, then we
replace $x$ by integers ranging from $-r$ to $r$, where $r=\lfloor \frac{k}{2} \rfloor$. Each interpolation
gives a linear equation with $k$ coefficients. Then we use the \emph{solve} command in {\tt Maple 17} to solve
these coefficients. Since this is a homogeneous linear system, the solution space should be $1$-dimensional.
This leads us to add another equation, for example $a_{1}=1$ where $a_{1}$ is one of the coefficients, to
guarantee the uniqueness of solution. After getting the solution, if all coefficients are rational, we
normalize them so that their content is $1$.

\begin{exa}\label{exa:Z}
Let $Z_{A,{\bf f}}$ be the almost-toric surface in $\PP^{3}$ parameterized by
\[
(s^{3}x^{2}(x-1):s^{2}t(x^{2}+1):st^{2}x(x+1)^{2}:t^{3}(x-1)(x-2)).
\]
Then
$
A=\begin{bmatrix}
3&2&1&0 \\
0&1&2&3 \\
\end{bmatrix}
$
and
$V_{f}=
\begin{bmatrix}
2&1&0&0&0&-3 \\
0&0&2&0&0&-2 \\
1&0&0&2&0&-3 \\
0&1&0&0&1&-2
\end{bmatrix}$. The vertices of ${\rm Newt}(Z_{A,{\bf f}}) $ are
\[
(6,0,0,6),(4,0,6,2),(2,2,8,0),(1,4,7,0),(0,7,4,1),(2,6,0,4).
\]
We find all monomials in the implicit polynomial to be
\[
\begin{split}
&u_{0}^6u_{3}^6, u_{0}^5u_{1}u_{2}u_{3}^5, u_{0}^5u_{2}^3u_{3}^4, u_{0}^4u_{1}^3u_{3}^5,
u_{0}^4u_{1}^2u_{2}^2u_{3}^4, u_{0}^4u_{1}u_{2}^4u_{3}^3, u_{0}^4u_{2}^6u_{3}^2, u_{0}^3u_{1}^4u_{2}u_{3}^4,
u_{0}^3u_{1}^3u_{2}^3u_{3}^3,\\ &u_{0}^3u_{1}^2u_{2}^5u_{3}^2, u_{0}^3u_{1}u_{2}^7u_{3},
u_{0}^2u_{1}^6u_{3}^4, u_{0}^2u_{1}^5u_{2}^2u_{3}^3, u_{0}^2u_{1}^4u_{2}^4u_{3}^2, u_{0}^2u_{1}^3u_{2}^6u_{3},
u_{0}^2u_{1}^2u_{2}^8, u_{0}u_{1}^6u_{2}^3u_{3}^2,\\ &u_{0}u_{1}^5u_{2}^5u_{3}, u_{0}u_{1}^4u_{2}^7,
u_{1}^7u_{2}^4u_{3}.
\end{split}
\]
We use interpolation to solve for the coefficients of these monomials and obtain the implicit polynomial
\[
\begin{split}
&8u_{0}^6u_{3}^6+52u_{0}^5u_{1}u_{2}u_{3}^5-28u_{0}^5u_{2}^3u_{3}^4-48u_{0}^4u_{1}^3u_{3}^5+58u_{0}^4u_{1}^2u_{2}^2u_{3}^4-82u_{0}^4u_{1}u_{2}^4u_{3}^3\\
&+25u_{0}^4u_{2}^6u_{3}^2-1312u_{0}^3u_{1}^4u_{2}u_{3}^4-175u_{0}^3u_{1}^3u_{2}^3u_{3}^3+476u_{0}^3u_{1}^2u_{2}^5u_{3}^2-50u_{0}^3u_{1}u_{2}^7u_{3}\\
&+72u_{0}^2u_{1}^6u_{3}^4-5760u_{0}^2u_{1}^5u_{2}^2u_{3}^3+5056u_{0}^2u_{1}^4u_{2}^4u_{3}^2-1194u_{0}^2u_{1}^3u_{2}^6u_{3}+25u_{0}^2u_{1}^2u_{2}^8\\
&-4176u_{0}u_{1}^6u_{2}^3u_{3}^2-3256u_{0}u_{1}^5u_{2}^5u_{3}-72u_{0}u_{1}^4u_{2}^7-576u_{1}^7u_{2}^4u_{3}.
\end{split}
\]
\end{exa}

\subsection*{Efficiency Test}
To test the efficiency of our algorithm, we try some simple examples using our implementation in {\tt Maple
17} and using ideal elimination in {\tt Macaulay 2}. The result is in Table \ref{tab:sim}:

\begin{table}[h]
\begin{tabular}{|c|c|c|c|c|}
  \hline
  sample & degree & \# of terms & our time cost & ideal elimination's time cost \\
  \hline
  Example \ref{exa:H} & 22 & 24 & 0.094s & 1.875s\\
  \hline
  Example \ref{exa:Z} & 10 & 16 & 0.047s & 0.078s.\\
  \hline
\end{tabular}
\caption{Simple Examples}\label{tab:sim}
\end{table}

We then try some examples that both {\tt Macaulay 2} and {\tt Sagemath} cannot solve in a reasonable time. We
generate some samples as the input using the following method: let $n$ be the dimension of the torus, $d$ the
degree of the homogeneous monomials and $k$ a positive integer. Then we choose $n+2$ degree $d$ monomials
randomly from all possible $\binom{n+d-1}{d}$ choices. For the univariate polynomials, we choose $n+2$
polynomials of the form $(x-2)^{*}(x-1)^{*}x^{*}(x+1)^{*}(x+2)^{*}$, where each $*$ is a random integer
between $0$ and $k$. Table \ref{tab:com} shows the time needed to find the implicit polynomial of almost-toric
hypersurfaces given by randomly generated inputs. It turns out that our implementation improves the efficiency
of finding the implicit polynomial of almost-toric hypersurfaces.

\begin{table}[hb]
\begin{tabular}{|c|c|c|c|c|c|c|c|}
  \hline
  sample & degree & \# of terms & time cost & sample & degree & \# of terms & time cost\\
  \hline
  1 & 213 & 109 & 12.484s & 6 & 179 & 97 & 8.110s \\
  \hline
  2 & 109 & 80 & 1.594s & 7 & 40 & 32 & 0.156s \\
  \hline
  3 & 172 & 129 & 10.421s & 8 & 27 & 14 & 0.140s \\
  \hline
  4 & 474 & 275 & 156.969s & 9 & 79 & 71 & 1.766s \\
  \hline
  5 & 291 & 137 & 20.375s & 10 & 281 & 148 & 20.719s \\
  \hline
\end{tabular}
\caption{$n=4,d=4,k=5$}\label{tab:com}
\end{table}


\begin{thebibliography}{99}

\bibitem[CLS11]{CLS11}
David~A Cox, John~B Little, and Henry~K Schenck, \emph{Toric varieties},
  American Math. Soc., 2011.

\bibitem[Cox03]{Cox03}
David Cox, \emph{What is a toric variety?}, Contemporary Mathematics
  \textbf{334} (2003), 203--224.

\bibitem[Cue10]{Cue10}
Maria~Angelica Cueto, \emph{Tropical implicitization}, Ph.D. thesis, University
  of California, Berkeley, 2010.

\bibitem[IS11]{IS11}
Nathan~Owen Ilten and Hendrik S{\"u}ss, \emph{Polarized complexity-1
  {$T$}-varieties}, Michigan Math. J. \textbf{60} (2011), no.~3, 561--578.

\bibitem[MS15]{MS14}
D.~Maclagan and B.~Sturmfels, \emph{Introduction to tropical geometry},
  Graduate Texts in Math., vol. ~161, American Math.~Soc., 2015.

\bibitem[PS08]{PS08}
Patrice Philippon and Mart{\'{\i}}n Sombra, \emph{A refinement of the {B}ern\v
  stein-{K}u\v snirenko estimate}, Adv. Math. \textbf{218} (2008), no.~5,
  1370--1418.

\bibitem[ST08]{ST08}
Bernd Sturmfels and Jenia Tevelev, \emph{Elimination theory for tropical
  varieties}, Math. Res. Lett. \textbf{15} (2008), no.~3, 543--562.

\bibitem[Stu96]{Stu96}
Bernd Sturmfels, \emph{Gr{\"o}bner bases and convex polytopes}, vol.~8,
  American Math. Soc., 1996.

\bibitem[STY07]{STY07}
Bernd Sturmfels, Jenia Tevelev, and Josephine Yu, \emph{The {Newton} polytope
  of the implicit equation}, Mosc. Math. J \textbf{7} (2007), no.~2, 327--346.

\bibitem[SY08]{SY08}
Bernd Sturmfels and Josephine Yu, \emph{Tropical implicitization and mixed
  fiber polytopes}, Software for algebraic geometry, Springer, 2008,
  pp.~111--131.

\end{thebibliography}
\end{document}